\def \A {\mathcal{A}}
\def \E {\mathbb{E}}
\def \R {\mathbb{R}}
\def \N {\mathbb{N}}
\def \T {\mathbb{T}}

\def \eps {\varepsilon}
\def \sbs {\subseteq}

\newcommand{\cco}[1] {\overline{\text{co}}\left(#1\right)}
\newcommand{\co}[1] {\text{co}\left(#1\right)}
\newcommand{\Ext}[1] {\text{Ext}\left(#1\right)}
\def \closure {\overline}

\documentclass{gen-j-l}

\newtheorem{theorem}{Theorem}[section]
\newtheorem{lemma}[theorem]{Lemma}
\newtheorem{corollary}[theorem]{Corollary}
\newtheorem{proposition}[theorem]{Proposition}

\theoremstyle{definition}
\newtheorem{definition}[theorem]{Definition}

\theoremstyle{remark}
\newtheorem{remark}[theorem]{Remark}

\numberwithin{equation}{section}

\begin{document}

\title{Nondentable Sets in Banach Spaces}


\author[Dilworth]{S. J. Dilworth}
\address{Department of Mathematics, University of South Carolina, Columbia, SC 29208, USA}
\curraddr{}
\email{dilworth@math.sc.edu}
\thanks{}

\author[Gartland]{Chris Gartland}
\address{Department of Mathematics, University of Illinois at Urbana-Champaign, Urbana, IL 61801, USA }
\curraddr{}
\email{cgartla2@illinois.edu}
\thanks{}

\author[Kutzarova]{Denka Kutzarova}
\address{Department of Mathematics University of Illinois at Urbana-Champaign Urbana, IL 61801, USA and Institute of Mathematics and Informatics, Bulgarian Academy of Sciences, Sofia, Bulgaria}
\curraddr{}
\email{denka@math.uiuc.edu}
\thanks{The third author was supported by Simons Foundation Collaborative Grant No 636954}

\author[Randrianarivony]{N. Lovasoa Randrianarivony}
\address{Department of Mathematics and Statistics, Saint Louis University, St. Louis, MO 63103, USA}
\curraddr{}
\email{nrandria@slu.edu}
\thanks{The fourth author was partially supported by NSF grant DMS-1301591}


\date{}

\dedicatory{}


\begin{abstract}
In his study of the Radon-Nikod\'ym property of Banach spaces, Bourgain showed (among other things) that in any closed, bounded, convex set $A$ that is nondentable, one can find a separated, weakly closed bush. In this note, we prove a generalization of Bourgain's result: in any bounded, nondentable set $A$ (not necessarily closed or convex) one can find a separated, weakly closed approximate bush. Similarly, we obtain as corollaries the existence of $A$-valued quasimartingales with sharply divergent behavior.
\end{abstract}

\maketitle




\section{Introduction}
We were motivated by the question of whether using the Kuratowski measure of noncompactness in place of diameter leads to a different  notion of dentability of (not necessarily closed or convex) subsets of $X$. Proposition \ref{prop:dentequiv} shows that they do not. This generalizes results from \cite[Chapitre 4]{B} where $A$ is assumed to be closed, bounded, and convex. In Section \ref{sec:sepbush}, we obtain as corollaries $A$-valued quasimartingales and $\cco{A}$-valued martingales with sharply divergent behavior (Corollaries \ref{cor:quasimartingale} and \ref{cor:martingale}) whenever $A$ is non-$\eps$-dentable. In Section \ref{sec:wcsepbush}, we improve the results of Section \ref{sec:sepbush} by showing that the range of the quasimartingale can be made weakly closed. As a further corollary, we show that one can find a countable set $F$ with $\lim_{F \ni f \to \infty} d(f,A) \to 0$ such that $\cco{F} \cap \Ext{\overline{co}^{w*}(F)} = \emptyset$ (Corollary \ref{cor:final}).

\section{Preliminaries}
For any topological vector space $V$ over $\R$ and $E \sbs V$, let $\co{E}$ denote the convex hull of $E$, and $\cco{E}$ the closure of $\co{E}$ in $V$. Henceforth, let $(X,\|\cdot\|)$ be a Banach space over $\R$. For $r > 0$ and $x \in X$, $B_r(x)$ denotes the open ball of radius $r$ centered at $x$. $B_X$ denotes the closed unit ball of $X$.

\begin{definition}
For any $A \sbs X$, let $\alpha(A)$ be the infimum over all $\eps > 0$ so that $A$ can be covered by finitely many sets of diameter at most $\eps$. $\alpha(A)$ is called the \emph{Kuratowksi measure of noncompactness} of $A$.
\end{definition}

\begin{definition}
For any bounded, nonempty $A \sbs X$, $f \in B_{X^*}$ (unit ball of $X^*$), and $\delta > 0$, we define the \emph{slice} $S(f,A,\delta)$, to be the set $\{a \in A: f(a) > \sup f(A) - \delta\}$. A \emph{slice of $A$} is a set $S(f,A,\delta)$ for some $f \in B_{X^*}$ and $\delta > 0$.
\end{definition}

\begin{remark} \label{rmk:sliceofcco}
Geometrically, a slice of $A$ is a nonempty intersection of $A$ with an open half-plane. Note that if $S(f,\cco{A},\delta)$ is a slice of $\cco{A}$, then $S(f,\cco{A},\delta) \cap A = S(f,A,\delta)$ is a slice of $A$. This is due to the fact that
$$\sup(f(\cco{A})) = \sup\left( \closure{f(\co{A})} \right) = \sup(f(\co{A})) = \sup(\co{f(A)}) = \sup(f(A))$$
\end{remark}

\begin{definition}
A bounded set $A \sbs X$ is called $\eps$\emph{-dentable} if there exists a slice of $A$ with diam$(A) \leq \eps$, and \emph{non-}$\eps$\emph{-dentable} otherwise. $A$ is \emph{dentable} if it is $\eps$-dentable for every $\eps > 0$, and \emph{nondentable} otherwise.
\end{definition}

\begin{remark} \label{rmk:nonepsdentable}
By Remark \ref{rmk:sliceofcco}, if $\cco{A}$ is $\eps$-dentable, $A$ is $\eps$-dentable.
\end{remark}

\begin{definition}
If $V$ is a topological vector space, $E \sbs V$ and $e \in E$, $e$ is called a \emph{denting point} of $E$ if $e \notin \cco{E \setminus U}$ for every neighborhood $U$ of $e$. Special cases are when $V$ is a Banach space equipped with the weak topology, or a dual Banach space equipped with the weak* topology, in which case we call $e$ a \emph{weak denting point} or a \emph{weak* denting point}, respectively.
\end{definition}

\begin{definition}
Let $\N^{<\omega}$ denote the set of finite length sequences of natural numbers. A \emph{tree} is a nonempty set $\T$ such that if $b \in \T$ and $b = (b',i)$ for some $b' \in \N^{<\omega}$ and $i \in \N$, then $b' \in \T$. In this case, $b$ is called a \emph{child} of $b'$. We say that $\T$ is \emph{finitely branching} if each $b \in \T$ has only finitely many children. If $b \in \T$ has $k$ children, we assume that they are $(b,1), \dots (b,k)$. Given a sequence $b \in \N^{<\omega}$, we let $|b|$ denote its length. For $n \in \N$, we let $\T_{\leq n} = \{b \in \T: |b| \leq n\}$, $\T_n = \{b \in \T: |b| = n\}$, and $\T_{\geq n} = \{b \in \T: |b| \geq n\}$. Given a positive sequence $(\delta_n)_{n \geq 0}$, finitely branching tree $\T$, and subset $(x_b)_{b \in \T} \sbs X$ indexed by $\T$, we say that $(x_b)_{b \in \T}$ is a $(\delta_n)_{n \geq 0}$-\emph{approximate bush} if for each $n \in \N$ and $b \in \T_n$ with children $(b,1), \dots (b,k)$, $b \in \co{x_{(b,1)}, \dots x_{(b,k)}} + B_{\delta_n}(0)$. If it always holds that $b \in \co{x_{(b,1)}, \dots x_{(b,k)}}$, then $(x_b)_{b \in \T}$ is a \emph{bush}. An approximate bush $(x_b)_{b \in \T}$ is $\delta$-\emph{separated} if for each $n \in \N$ and $b \in \T_n$ and child $(b,i)$, $\|x_b - x_{(b,i)}\| > \delta$.
\end{definition}

\begin{definition}
Given a filtration $(\A_n)_{n \geq 0}$ and a positive sequence $(\delta_n)_{n \geq 0}$, we say that a sequence of $X$-valued, $(\A_n)_{n \geq 0}$-adapted random variables $(M_n)_{n \geq 0}$ is a $(\delta_n)_{n \geq 0}$-\emph{quasimartingale} if
$$\|\E(M_{n+1} | \A_n) - M_n\|_{\infty} \leq \delta_n$$
for all $n \geq 0$. If $\|\E(M_{n+1} | \A_n) - M_n\|_{\infty} = 0$ always holds, $(M_n)_{n \geq 0}$ is a \emph{martingale}.
\end{definition}

The following proposition can be found in \cite[Lemme 4.2]{B}. For the sake of self-containment, we include our own proof here.

\begin{proposition} \label{prop:lemme4.2}
Let $\eps>0$ and $\delta>0$. Suppose that  $C$ and $C_1$   are closed,  bounded, convex sets with  $C_1$ properly contained in $C$. If $C =\overline{ \operatorname{co}}(C_1 \cup C_2)$, where $C_2$ is a  convex  subset  of $C$ and $\operatorname{diam}(C_2)<\eps$,  then
there exists a slice $S$ of $C$ with $S \subseteq C_2 + B_\delta(0)$. In particular, $C$ is $\eps$-dentable.
\end{proposition}

\begin{proof}
We may assume that $\operatorname{diam}(C) \le 1$. Since $C_1$ is a proper convex  subset of $C$, by Hahn-Banach separation there exists $f \in B_{X^*}$ such that 
$$\sup f(C_1) < M := \sup f(C)$$
Hence $C_1 \subseteq C\setminus S(f, C, \alpha)$ for some $\alpha>0$. So
$$C = \overline{\operatorname{co}}(( C\setminus S(f, C, \alpha)) \cup C_2)$$
For $\gamma>0$, let $S_\gamma = S(f,C,\gamma)$. Consider $y \in S_\gamma$. There exist $\lambda \in [0,1]$, $z_1\in \operatorname{co}( C\setminus S(f, C, \alpha))$, and $z_2 \in C_2$ such that $\|y - \lambda z_1 -(1-\lambda)z_2\| < \gamma$. Hence
\begin{align*}
M - \gamma &< f(y) \\
& \le f(\lambda z_1 +(1-\lambda)z_2) +\|y - \lambda z_1 -(1-\lambda)z_2\|\\
&< \lambda f(z_1) + (1-\lambda)f(z_2) + \gamma\\ &\le \lambda(M - \alpha) + (1-\lambda)M + \gamma\\&= M - \lambda \alpha + \gamma.
\end{align*}
Hence $\lambda < 2\gamma/\alpha$. So
$$\|y - z_2\| < \lambda\|z_1 - z_2\| + \gamma \le (2\gamma/\alpha) \operatorname{diam}(C) + \gamma \le \gamma(2/\alpha + 1)$$
So, setting $\gamma := \frac{\delta\alpha}{2+\alpha}$, we get $S:= S_\gamma \subseteq  C_2 + B_\delta(0)$. Note that $\operatorname{diam}(S) \le \operatorname{diam}(C_2)+ 2\delta < \eps$ for $\delta$ sufficiently small. So $C$ is $\eps$-dentable.
\end{proof}

We now derive a corollary of this proposition that will play a crucial role in the proof of Lemma \ref{lem:2}.

\begin{corollary} \label{cor:lemme4.2}
For any closed, bounded, convex, non-$\eps$-dentable $C \sbs X$, any closed, convex $C' \sbs C$, and any $D \sbs C$ with $\alpha(D) < \eps$, if $C = \cco{C' \cup D}$, then $C = C'$.
\end{corollary}

\begin{proof}
Let $C$, $C'$, and $D$ be as above. Assume $C = \cco{C' \cup D}$. Since $\alpha(D) < \eps$, $D = B_1 \cup B_2 \cup \dots B_n$ for some $B_i \sbs D$ with diam$(B_i) < \eps$. Let $C_i = \cco{B_i}$. Then diam$(C_i) =$ diam$(B_i) < \eps$, and $C = \cco{C' \cup C_1 \cup C_2 \cup \dots C_n}$. Since $C$ is closed, bounded, convex, and not $\eps$-dentable, and since $C_n \sbs C$ is closed, convex with diam$(C_n) < \eps$, Proposition \ref{prop:lemme4.2} (with $C_2 = C_n$ and $C_1 = \cco{C' \cup C_1 \cup C_2 \cup \dots C_{n-1}}$) implies that $C = \cco{C' \cup C_1 \cup C_2 \cup \dots C_{n-1}}$. Since diam$(C_{n-1}) < \eps$, we may apply Proposition \ref{prop:lemme4.2} again to obtain \\ $C = \cco{C' \cup C_1 \cup C_2 \cup \dots C_{n-2}}$. Iterating, we get $C = C'$.
\end{proof}

\section{$\delta$-Separated Martingales and Bushes} \label{sec:sepbush}
\begin{proposition} \label{prop:dentequiv}
Let $A \sbs X$ be bounded, and let $\eps > 0$. The following are equivalent:
\begin{enumerate}
\item $\alpha(S) \geq \eps$ for every slice $S \sbs A$.
\item diam$(S) \geq \eps$ for every slice $S$ of $A$ ($A$ is non-$\eps$-dentable).
\item diam$(S) \geq \eps$ for every slice $S$ of $\cco{A}$ ($\cco{A}$ is non-$\eps$-dentable).
\end{enumerate}
\end{proposition}

\begin{proof}
Let $A$, $\eps$ be as above. $(1) \to (2)$ is clear from definition of $\alpha$. $(2) \to (3)$ follows from the fact that every slice of $\cco{A}$ contains a slice of $A$. We now show $(3) \to (1)$ by contradiction. Let $C = \cco{A}$, assume that $C$ is non-$\eps$-dentable and that there exists a slice $S = S(f,A,\delta)$ of $A$ with $\alpha(S) < \eps$. Set $S_C = S(f,C,\delta)$. Then since $C \setminus S_C$ is a closed convex subset of $C$ and $C = \cco{(C \setminus S_C) \cup S}$. Then Corollary \ref{cor:lemme4.2} implies $C = C \setminus S_C$, a contradiction since $S_C \sbs C$ and $S_C$ is nonempty.
\end{proof}

As in \cite[Chapitre 4]{B}, we obtain several corollaries.

\begin{corollary} \label{cor:1}
For any $A \sbs X$ bounded and $\eps > 0$, if $A$ is non-$\eps$-dentable, then for all $\delta < \frac{\eps}{2}$ and all $a_1, a_2, \dots a_n \in A$, $\cco{A} = \cco{A \setminus (B_\delta(a_1) \cup B_\delta(a_2) \cup \dots B_\delta(a_n))}$.
\end{corollary}

\begin{proof}
Let $A$, $\eps$, $\delta$, and $a_1,a_2, \dots a_n$ be as above. Suppose there exists $x \in \cco{A} \setminus \cco{A \setminus (B_\delta(a_1) \cup B_\delta(a_2) \cup \dots B_\delta(a_n))}$. By Hahn-Banach separation, we can pick a slice $S$ of $\cco{A}$ containing $x$ and disjoint from \\
$\cco{A \setminus (B_\delta(a_1) \cup B_\delta(a_2) \cup \dots B_\delta(a_n))}$. Then $S \cap A$ is a slice of $A$ disjoint from $A \setminus (B_\delta(a_1) \cup B_\delta(a_2) \cup \dots B_\delta(a_n))$, and thus $S \cap A \sbs B_\delta(a_1) \cup B_\delta(a_2) \cup \dots B_\delta(a_n)$, which implies $\alpha(S \cap A) \leq 2\delta < \eps$, contradicting Proposition \ref{prop:dentequiv}.
\end{proof}

We can use Corollary \ref{cor:1} to construct $A$-valued quasimartingales and $\cco{A}$-valued martingales that diverge in a sharp manner.

\begin{corollary} \label{cor:quasimartingale}
For any nonempty, bounded, non-$\eps$-dentable $A \sbs X$, any $\delta < \frac{\eps}{2}$, and any positive, summable sequence $(\delta_n)_{n \geq 0}$, there exists a filtration of finite $\sigma$-algebras $(\mathcal{A}_n)_{n \geq 0}$ on $[0,1]$, each of whose atoms are intervals, and an $(\A_n)_n$-adapted sequence of random variables $(M_n)_{n \geq 0}$ such that, for all $s,t \in [0,1]$ and $m \neq n \geq 0$,
\begin{enumerate}
\item $M_n$ takes values in $A$.
\item $\|M_n(s) - M_m(t)\| > \delta$.
\item $(M_n)_{n \geq 0}$ is a $(\delta_n)_{n \geq 0}$-quasimartingale: $\|\E(M_{n+1} | \A_n) - M_n\|_\infty < \delta_n$.
\end{enumerate}
\end{corollary}

\begin{proof}
Let $A \sbs X$ and $\delta > 0$ be as above. We construct the martingale inductively. Let $x_0$ be any point of $A$, $\A_0$ the trivial $\sigma$-algebra, and $M_0 \equiv x_0$. Suppose that, for some $N \in \N$, $\A_n$ and $M_n$ have been constructed for all $n \leq N$ and satisfy the conclusion of the Corollary \ref{cor:quasimartingale}. Let $J$ be an atom of $\A_N$, and let $x_J$ be the value of $M_N$ on $J$. Let $\{a_1, a_2, \dots a_k\} \sbs A$ be the set of all elements in the image of any one of the $M_n$, $n \leq N$. By Corollary \ref{cor:1}, $x_J \in \cco{A \setminus (B_\delta(a_1) \cup B_\delta(a_2) \cup \dots B_\delta(a_k))}$. Thus, there exists $z_J \in \co{A \setminus (B_\delta(a_1) \cup B_\delta(a_2) \cup \dots B_\delta(a_k))}$ such that $\|x_J - z_J\| < \delta_N$. Since $z_J \in \co{A \setminus (B_\delta(a_1) \cup B_\delta(a_2) \cup \dots B_\delta(a_k))}$, $z_J = \lambda_1 z^1_J + \lambda_2 z^2_J + \dots \lambda_m z^m_J$ for some $z^1_J, z^2_J, \dots z^m_J \in A$ and $\lambda_1, \lambda_2, \dots \lambda_m \in (0,1)$ with $\lambda_1 + \lambda_2 + \dots \lambda_m = 1$ and $\|z^i_J - a_j\| > \delta$ for all $i \leq m$ and $j \leq k$. Now we subdivide the interval $J$ into $m$ pairwise disjoint subintervals, $J_1, J_2, \dots J_m$, with $|J_i| = \lambda_i|J|$ for each $i$. Repeating this process for each atom $J \in \A_N$ gives us a collection of pairwise disjoint intervals, and we define $\A_{N+1}$ to be the $\sigma$-algebra that they generate. On each $J_i$, we define $M_{N+1}$ to be $z^i_J$. Then conclusions (1) and (2) hold, and (3) holds since $\|\E(M_{N+1}|\A_N)-M_N\|_\infty = \sup_{J,i} \|z_J-z^i_J\| < \delta_N$.
\end{proof}

\begin{corollary} \label{cor:martingale}
For any nonempty, bounded, non-$\eps$-dentable $A \sbs X$, any $\delta < \frac{\eps}{2}$, and any positive, summable sequence $(\delta_n)_{n \geq 0}$, there exists a filtration of finite $\sigma$-algebras $(\mathcal{A}_n)_{n \geq 0}$ on $[0,1]$, each of whose atoms are intervals, an $(\A_n)_{n \geq 0}$-adapted quasimartingale $(M_n)_{n \geq 0}$, and an $(\A_n)_{n \geq 0}$-adapted martingale $(\overline{M}_n)_{n \geq 0}$ such that, for all $s,t \in [0,1]$ and $m \neq n \geq 0$,
\begin{enumerate}
\item $M_n$ takes values in $A$.
\item $\overline{M}_n$ takes values in $\cco{A}$.
\item $\|M_n - \overline{M}_n\|_\infty < \delta_n$.
\item $\|M_n(s) - M_m(t)\|, \|\overline{M}_n(s) - \overline{M}_m(t)\| > \delta$.
\end{enumerate}
\end{corollary}

\begin{proof}
Let $A \sbs X$, and $\delta > 0$ be as above. Choose $\delta' \in (\delta,\frac{\eps}{2})$ and assume $\sum_{n=0}^\infty \delta_n < \delta' - \delta$. Choose a positive sequence $(\gamma_k)_{k \geq 0}$ such that $\sum_{k=n}^\infty \gamma_k < \delta_n$, and note that this implies $\sum_{n=0}^\infty \gamma_n < \sum_{n=0}^\infty \delta_n < \delta' - \delta$. By Corollary \ref{cor:quasimartingale}, there is a filtration $(\A_n)_{n \geq 0}$ and an $A$-valued, $(\A_n)_{n \geq 0}$-adapted quasimartingale $(M_n)_{n \geq 0}$ such that $\|M_n(s) - M_m(t)\| > \delta'$ for all $s,t \in [0,1]$, $m \neq n$, and $\|\E(M_{n+1}|\A_n) - M_n\|_\infty < \gamma_n$. This inequality, together with the fact that $(\delta_n)_{n \geq 0}$ is summable (and thus convergent to 0), implies, for each $n \geq 0$, the sequence $(\E(M_{k}|\A_n))_{k \geq n}$ is Cauchy in $L^\infty(I;X)$. Indeed, for $k > j \geq n$,
$$\|\E(M_k - M_j|\A_n)\|_{L^\infty(I;X)} \leq \sum_{r=j}^{k-1} \|\E(M_{r+1} - M_r|\A_n)\|_{L^\infty(I;X)}$$
$$\leq \sum_{r=j}^{k-1} \|\E(M_{r+1} - M_r|\A_r)\|_{L^\infty(I;X)} = \sum_{r=j}^{k-1} \|\E(M_{r+1}|\A_r) - M_r\|_{L^\infty(I;X)}$$
$$\leq \sum_{r=j}^{k-1} \gamma_r \leq \delta_j$$
Thus we may set $\overline{M}_n := \lim_{k \to \infty} \E(M_{k}|\A_n)$. Clearly, $(\overline{M}_n)_{n \geq 0}$ is adapted to $(\A_n)_{n \geq 0}$ and takes values in $\cco{A}$, showing (2). Let us check the martingale property: 
$$\E(\overline{M}_{n+1} | \A_n) = \E(\lim_{k \to \infty} \E(M_{k}|\A_{n+1}) | \A_n) = \lim_{k \to \infty} \E(\E(M_{k}|\A_{n+1}) | \A_n)$$
$$= \lim_{k \to \infty} \E(M_{k}| \A_n) = \overline{M}_{n+1}$$
showing (1). Next,
$$\|\overline{M}_n - M_n\|_\infty \leq \sum_{k=n}^\infty \|\E(M_{k+1}-M_k|\A_n)\|_\infty \leq \sum_{k=n}^\infty \|\E(M_{k+1}-M_k|\A_k)\|_\infty$$
$$= \sum_{k=n}^\infty \|\E(M_{k+1}|\A_k) - M_k\|_\infty \leq \sum_{k=n}^\infty \gamma_k < \delta_n$$
showing (3). We then use (3) to show (4):
$$\|\overline{M}_n(s) - \overline{M}_m(t)\| \geq \|M_n(s) - M_m(t)\| - \delta_n - \delta_m > \delta' - (\delta' - \delta) = \delta$$
\end{proof}

\begin{remark}
The union over $n$ of the image of $M_n$ forms a $\delta$-separated bush in $\cco{A}$. It  is norm closed and lacks extreme points.
\end{remark}

\section{Weakly Closed $\delta$-separated Martingales and Bushes} \label{sec:wcsepbush}
In this section, we sharpen our results from the previous section by constructing an $A$-valued $\delta$-separated approximate bush that is weakly closed. The argument is more involved than those of the previous section. This again extends results from Bourgain in \cite{B}. $A$ is not assumed to be closed or convex in our case. 

\begin{definition}
Let $A \sbs B_X$ and let $C = \cco{A}$. For any $\gamma \in (0,1)$ and slice $S = S(f,C,\delta)$ of $C$, we define $S^\gamma = S\left(f,C,\frac{\gamma\delta}{2}\right)$. $S^\gamma$ is called a $\boldsymbol{\gamma}$\textbf{-shallow parallel of} $\boldsymbol{S}$.
\end{definition}

\begin{lemma} \label{lem:1}
For any $C \sbs B_X$ closed and convex, any $\gamma \in (0,1)$, and any slice $S$ of $C$, $S^\gamma \sbs S$. For any $E \sbs C$ for which $C = \cco{(C \setminus S) \cup E}$, $S^\gamma \sbs \cco{E} + \closure{B}_\gamma(0) \sbs \co{E} + B_{2\gamma}(0)$.
\end{lemma}

\begin{proof}
Let $\gamma \in (0,1)$ and $S = S(f,C,\delta)$ a slice of $C$. Since $\gamma \in (0,1)$, $\frac{\gamma\delta}{2} < \delta$ implying $S^\gamma = S\left(f,C,\frac{\gamma\delta}{2}\right) \sbs S(f,C,\delta) = S$. For the second part, let $E \sbs C$ such that $C = \cco{(C \setminus S) \cup E}$. Let $y \in S^\gamma$, $\eps > 0$, and $M := \sup(f(C))$. Since $y \in C = \cco{(C \setminus S) \cup E}$, there exist $\lambda \in [0,1]$, $z_1 \in (C \setminus S)$, $z_2 \in \cco{E}$, and $u \in X$ with $\|u\| < \eps$ such that $y = \lambda z_1 + (1-\lambda)z_2 + u$. Then we have

$\begin{array}{lll}
M - \frac{\gamma\delta}{2} & < & f(y) \\
& = & \lambda f(z_1) + (1-\lambda) f(z_2) + f(u) \\
& < & \lambda (M - \delta) + (1-\lambda)M + \eps
\end{array}$

implying $\lambda < \frac{\gamma}{2} + \frac{\eps}{\delta}$. Hence,

$\begin{array}{lll}
\|y-z_2\| & \leq & \|y - (1-\lambda)z_2\| + \|(1-\lambda)z_2-z_2\| \\
& =      & \|\lambda z_1 + u\| + \|\lambda z_2\| \\
& \leq   & 2\lambda + \eps \\
& \leq   & \gamma + \frac{2\eps}{\delta} + \eps
\end{array}$

Since $\eps > 0$ was arbitrary, this shows $y \in B_\gamma(z_2) \sbs \cco{E} + \closure{B}_\gamma(0)$. The final containment $\cco{E} + \closure{B}_\gamma(0) \sbs \co{E} + B_{2\gamma}(0)$ obviously holds.
\end{proof}

\begin{lemma} \label{lem:2}
Let $A \sbs X$ be bounded, nonempty, and non-$\eps$-dentable, and let $C = \cco{A}$ (by Remark \ref{rmk:nonepsdentable}, $C$ is non-$\eps$-dentable). For any slice $S_0$ of $C$, $D \sbs C$ with $\alpha(D) < \eps$, and $\gamma \in (0,1)$, let $\mathcal{S}(S_0,D)$ be the collection of all slices $S$ of $C$ with $S \sbs S_0 \setminus D$ and $\mathcal{S}^\gamma(S_0,D) = \{S^\gamma\}_{S \in \mathcal{S}(S_0,D)}$. Let $\Lambda = \Lambda(S_0,D,\gamma) \sbs C$ denote the union of all sets in $\mathcal{S}^\gamma(S_0,D)$. Then $C = \cco{(C \setminus S_0) \cup (\Lambda \cap A)}$.
\end{lemma}

\begin{proof}
Let $S_0$, $D$, $\gamma$, and $\Lambda$ be as above. By Corollary \ref{cor:lemme4.2} (with \\ $C' = \cco{(C \setminus S_0) \cup (\Lambda \cap A)}$ and $D = D$), it suffices to prove $C =$ \\ $\cco{(C \setminus S_0) \cup D \cup (\Lambda \cap A)}$. Assume $C \neq \cco{(C \setminus S_0) \cup D \cup (\Lambda \cap A)}$. Then by Hahn-Banach separation, there exists a slice $S$ of $C$ such that $S \sbs$ \\ $C \setminus \cco{(C \setminus S_0) \cup D \cup (\Lambda \cap A)}$. This implies $S \sbs S_0$, $S \cap D = \emptyset$, and $S \cap (\Lambda \cap A) = \emptyset$. Then $S \sbs S_0 \setminus D$. Thus, $S \in \mathcal{S}(S_0,D)$, so $S^\gamma \in \mathcal{S}^\gamma(S_0,D)$, and finally $S^\gamma \sbs \Lambda$. But since we also have $S^\gamma \sbs S$ and $S \cap (\Lambda \cap A) = \emptyset$, $(S^\gamma \cap A) = S^\gamma \cap (\Lambda \cap A) = \emptyset$, a contradiction since $S^\gamma \cap A$ is a slice of $A$ (since $S^\gamma$ is a slice of $C = \cco{A}$) and slices of nonempty sets are nonempty.
\end{proof}

\subsection{The Construction}
\begin{theorem} \label{thm:constr}
Let $A \sbs B_X$ be nonempty and non-$\eps$-dentable (not necessarily closed or convex), and $C = \cco{A}$ so that $C$ is also non-$\eps$-dentable. Fix $\delta < \frac{\eps}{2}$, and assume that $A$ is separable. Then $C$ is separable as well, so $C = \bigcup_{i=0}^\infty B_i$ for some open $B_i$ (relative to $C$) with diam$(B_i) < \eps$. Let $(\delta_n)_{n \geq 0}$ be a sequence of numbers in $(0,1)$. There exist a finitely branching tree $\T \sbs \N^{<\omega}$, an $(2\delta_n)_{n \geq 0}$-approximate bush $(x_b)_{b \in \T} \sbs A$, and slices $(S_b)_{b \in \T}$ of $C$ such that, for all $n \in \N$,

\begin{enumerate}
\item \label{prt:1} For all $b \in \T_n$, $x_b \in S_b^{\delta_n} \cap A \sbs S_b$.
\item \label{prt:2} If $n \geq 1$, then for all $b \in \T_n$, $S_b \cap B_{n-1} = \emptyset$ and $S_b \cap \left(\bigcup_{|p| \leq {n-1}} B_\delta(x_p)\right) = \emptyset$.
\item \label{prt:3} If $n \geq 1$, then for all $b \in \T_{n-1}$, if $(b,1), \dots (b,q)$ are the children of $b$, then $S_{(b,i)} \sbs S_b$ and the approximate bush property is satisfied: $x_b \in \text{co}(x_{(b,1)}, \dots x_{(b,q)}) + B_{2\delta_{n-1}}(0)$.
\end{enumerate}
\end{theorem}

\begin{proof}
The proof is by induction on $n$. For the base case, let $S_\emptyset = C$ and let $x_\emptyset$ be any element of $S_\emptyset^{\delta_0}$. For the inductive step, let $n \geq 0$ and assume $\T_{\leq n}$, $(x_b)_{b \in \T_{\leq n}} \sbs A$, and $(S_b)_{b \in \T_{\leq n}} \sbs C$ have been constructed, and satisfy (\ref{prt:1})-(\ref{prt:3}). Let $b \in \T_{n}$. Let $D := B_n \cup \bigcup_{|p| \leq {n}} B_\delta(x_p)$, so that $\alpha(D) < \eps$. As in Lemma \ref{lem:2}, let $\mathcal{S}(S_b,D)$ be the collection of all slices $S$ of $C$ such that $S \sbs S_b \setminus D$, $\mathcal{S}^{\delta_{n+1}}(S_b,D) = \{S^{\delta_{n+1}}\}_{S \in \mathcal{S}(S_0,D)}$, and $\Lambda = \bigcup \mathcal{S}^{\delta_{n+1}}(S_b,D)$. By Lemma \ref{lem:2}, $C = \cco{(C \setminus S_b) \cup (\Lambda \cap A)}$. Then by Lemma \ref{lem:1}, $S_b^{\delta_n} \sbs \co{\Lambda \cap A} + B_{2\delta_n}(0)$. Then since $x_b \in S_b^{\delta_n}$, there exists $z \in \co{\Lambda \cap A}$ such that $\|x_b - z\| < 2\delta_n$. Let $z_1, \dots z_q \in \Lambda \cap A$ and $\lambda_1^b, \dots \lambda_q^b \in [0,1]$ such that $z = \lambda_1^b z_1 + \dots \lambda_q^b z_q$. For each $i = 1, \dots q$, since $z_i \in \Lambda$, there are slices $S_{z_i} \in \mathcal{S}(S_b,D)$ of $C$ with $z_i \in S_{z_i}^{\delta_{n+1}}$, by definition of $\Lambda$. We now define the children of $b$ to be $(b,1), \dots (b,q)$, $x_{(b,i)}$ to be $z_i$, and $S_{(b,i)}$ to be $S_{z_i}$. Repeating this process for each $b \in \T_n$ gives us $\T_{n+1}$, $(x_b)_{b \in \T_{n+1}} \sbs A$, and $(S_b)_{b \in \T_{n+1}} \sbs C$.

(\ref{prt:1}) and (\ref{prt:3}) hold immediately by construction. It is also clear that (\ref{prt:2}) holds by recalling that $S_{(b,i)} \in \mathcal{S}(S_b,D)$, and then examining the definition of $D$ and $\mathcal{S}(S_b,D)$.
\end{proof}

\begin{remark} \label{rmk:nonseparable}
The assumption that $A$ is separable can be removed (at the penalty of replacing $\eps$ by $\eps/2$) because of the following result: under the hypothesis of Theorem \ref{thm:constr}, $A$ contains a countable subset that is non-$\eps/2$-dentable. This is essentially proved in \cite[Lemma 2.2]{M}, but we'll include the argument here. Since $\operatorname{diam}(S) > \varepsilon$ for every slice $S$ of $A$, it follows that no slice is contained in a closed ball 
$B_{\varepsilon/2}(x)$. Hence, if $a \in A$, then $a \in  \overline{\operatorname{co}}(A \setminus B_{\varepsilon/2}(a))$. So there exists a countable set 
$T(a) \subseteq A\setminus B_{\varepsilon/2}(a)$ such that $a \in \overline{\operatorname{co}}(T(a))$. By applying this fact iteratively as in  \cite[Lemma 2.2]{M}, we can construct a countable $A_0 \subseteq A$ such that for every $a \in A_0$, we have $a \in \overline{\operatorname{co}}(A_0 \setminus B_{\varepsilon/2}(a))$. Hence every slice $S$ of $A_0$ satisfied $\operatorname{diam}(S) > \varepsilon/2$. Hence $A_0$ is not $\eps/2$-dentable.

\end{remark}

\begin{corollary} \label{cor:wcbush}
For any separable $A \sbs B_X$ nonempty and non-$\eps$-dentable, any $\delta < \frac{\eps}{2}$, and any positive $(\delta_n)_{n \geq 0}$, there exists a $\delta$-separated, $(\delta_n)_{n-0}^\infty$-approximate bush $(x_b)_{b \in \T}$ in $A$ such that any other set $(y_b)_{b \in \T} \sbs C = \cco{A}$, with $\sup_{b \in \T_n}\|y_b - x_b\| < \gamma_n$ for some $\gamma_n \to 0$, is weakly closed and discrete. In particular, $(x_b)_{b \in \T}$ is weakly closed and discrete.
\end{corollary}

\begin{proof}
Let $A$, $\delta$, $(\delta_n)_{n \geq 0}$ be as above. Applying the construction of Theorem \ref{thm:constr}, with $(\delta_n/2)_{n \geq 0}$ in place of $(\delta_n)_{n \geq 0}$, yields a bush $(x_b)_{b \in \T}$. By Theorem \ref{thm:constr}(\ref{prt:1}), $x_b \in A$ for all $b \in \T$. Suppose $b_1,b_2 \in \T$ with $|b_2| > |b_1|$. Then by Theorem \ref{thm:constr}(\ref{prt:1}), $x_{b_2} \in S_{b_2}$, and by Theorem \ref{thm:constr}(\ref{prt:2}), $S_{b_2} \cap B_\delta(x_{b_1}) = \emptyset$, so $\|x_{b_2}-x_{b_1}\| > \delta$. This means the bush is $\delta$-separated. By Theorem \ref{thm:constr}(\ref{prt:3}), if $b \in \T$ and $(b,1), \dots (b,q)$ are the children of $b$, then $x_b \in \text{co}(x_{(b,1)}, \dots x_{(b,q)}) + B_{\delta_n}(0)$. This means the bush is $(\delta_n)_{n \geq 0}$-approximate.

Finally, let $(y_b)_{b \in \T} \sbs C$, with $\sup_{b \in \T_n}\|y_b - x_b\| < \gamma_n$ for some $\gamma_n \to 0$, and let $z$ belong to the weak closure of $(y_b)_{b \in \T}$. Since $C$ is norm closed and convex, it is weakly closed, and thus $z \in C$. Then $z \in B_i$ for some $i$. Consider $S_b$ for $|b| = i+1$. Then $S_b =S(f_b,C,\alpha_b)$ for some $f_b \in B_{X^*}$ and $\alpha_b > 0$. Hence
$$z \in B_i \subseteq C \setminus S_b = \{x \in C : f_b(x) \le \sup f(C) - \alpha_b\}$$
Since $B_i$ is open in the norm topology relative in $C$ and $C$ is convex, it follows that $B_i \subseteq \{x \in C : f_b(x) < \sup f_b(C) - \alpha_b\}$. Since $\gamma_n \to 0$, we can find $\gamma > 0$ and $N$ large enough so that $B_i \subseteq \{x \in C : f_b(x) < \sup f_b(C) - \alpha_b - \gamma\}$, $N \geq i+1$, and $\gamma_n < \gamma$ for all $n \geq N$. Then we set $U_b := \{x \in C : f_b(x) < \sup f_b(C) - \alpha_b - \gamma\}$ and observe that it is a weak neighborhood of $z$ in $C$. Hence $U :=  \cap_{|b| = i+1} U_b$ is a weak neighborhood of $z$ in $C$. Now we wish to show the set $U \cap (y_b)_{b \in \T}$ is finite, which will imply our desired conclusion that $(y_b)_{b \in \T}$ is weakly closed and discrete. We will show that $U \cap (y_b)_{b \in \T}$ is finite by showing that $U \cap (y_b)_{b \in \T_{\geq N}} = \emptyset$. Consider $b \in \T$ with $|b| \geq N$. Then $\|y_b - x_b\| < \gamma_{|b|} < \gamma$. Let $b_{i+1} \in \T$ denote the unique predecessor of $b$ with $|b_{i+1}|=i+1$. Then $x_b \in S_b \sbs S_{b_{i+1}}$, and hence $f_{b_{i+1}}(x_b) > \sup f_{b_{i+1}}(C) - \alpha_{b_{i+1}}$. Since $f_{b_{i+1}} \in B_{X^*}$ and $\|y_b - x_b\| < \gamma$, this implies $f_{b_{i+1}}(y_b) > \sup f_{b_{i+1}}(C) - \alpha_{b_{i+1}} - \gamma$. Thus, by definition of $U_{b_{i+1}}$, $y_b \notin U_{b_{i+1}}$. By definition of $U$ this proves $U \cap (y_b)_{b \in \T_{\geq N}} = \emptyset$.
\end{proof}

\begin{corollary} \label{cor:wcquasimartingale}
For any $A \sbs B_X$ nonempty and non-$\eps$-dentable, any $\delta < \frac{\eps}{2}$, and any positive sequence $(\delta_n)_{n \geq 0}$, there exists a filtration of finite $\sigma$-algebras $(\A_n)_{n \geq 0}$, an $A$-valued, $(\A_n)_{n \geq 0}$-adapted $(\delta_n)_{n \geq 0}$-quasimartingale $(M_n)_{n \geq 0}$ with $\|M_n(s) - M_m(t)\| > \delta$ for all $n \geq m \geq 0$ and $s,t \in [0,1]$, and the range of this quasimartingale is weakly closed and discrete.
\end{corollary}

\begin{proof}
Let $A$, $\delta$, $(\delta_n)_{n \geq 0}$ be as above, and apply Corollary \ref{cor:wcbush} to obtain a $(\delta_n)_{n \geq 0}$-approximate bush $(x_b)_{b \in \T}$ which is weakly closed and discrete. We define the filtration $(\A_n)_{n \geq 0}$ on $[0,1]$ recursively: Let $\A_0$ be the trivial $\sigma$-algebra. Suppose $\A_n$ has been defined as a finite whose atoms are intervals, the atoms are in bijection with $\T_{n}$ via $b \mapsto I_b$, and for any $b \in \T_{n-1}$ and child $(b,i) \in \T_n$, $\mathcal{L}(I_{(b,i)}) = \mathcal{L}(I_b)\lambda^b_i$. Then for any $b' \in \T_n$ with children $(b',1), \dots (b',q)$, we pick any subdivision of $I_{b'}$ into intervals $I_{(b',1)}, \dots I_{(b',q)}$ so that $\mathcal{L}(I_{(b',i)}) = \mathcal{L}(I_{b'})\lambda^{b'}_i$. Take $\A_{n}$ to be the $\sigma$-algebra generated by these intervals. Then we define $M_n$ to be $\sum_{|b|=n} x_b \chi_{I_b}$. We then have $\|\E(M_{n+1} | \A_n) - M_n\|_{L^\infty}$ = $\sup_{b \in \T_n} \|x_b -\lambda_1x_{(b,1)} - \dots \lambda_qx_{(b,q)}\| < \delta_n$. The range of this quasimartingale is exactly the bush, and thus weakly closed and discrete.
\end{proof}

\begin{corollary} \label{cor:wcmartingale}
For any $A \sbs B_X$ nonempty and non-$\eps$-dentable, $\delta < \frac{\eps}{2}$, and positive, summable sequence $(\delta_n)_{n \geq 0}$, there exist a filtration of finite $\sigma$-algebras $(\A_n)_{n \geq 0}$, an $A$-valued, $(\A_n)_{n \geq 0}$-adapted $(\delta_n)_{n \geq 0}$-quasimartingale $(M_n)_{n \geq 0}$ and \\ $\cco{A}$-valued, $(\A_n)_{n \geq 0}$-adapted martingale $(\overline{M}_n)_{n \geq 0}$ with, for all $n \neq m \geq 0$ and $s,t \in [0,1]$,
\begin{enumerate}
\item $\|\overline{M}_n(s) - \overline{M}_m(t)\| > \delta$.
\item $\|M_n - \overline{M}_n\|_\infty < \delta_n$.
\item The range of $(\overline{M}_n)_{n \geq 0}$ is weakly closed and discrete.
\end{enumerate}
\end{corollary}

\begin{proof}
Let $A$, $\delta$, $(\delta_n)_{n \geq 0}$ be as above, and apply Corollary \ref{cor:wcquasimartingale} to obtain the $\sigma$-algebra $(\A_n)_{n \geq 0}$ and $A$-valued, $(\delta_n)_{n \geq 0}$-quasimartingale $(M_n)_{n \geq 0}$ with weakly closed and discrete range. Construct $(\overline{M}_n)_{n \geq 0}$ from $(M_n)_{n \geq 0}$ just as in the proof of Corollary \ref{cor:martingale}, so that $(\overline{M_n})_{n \geq 0}$ is $\cco{A}$-valued and (1) and (2) hold. To see (3), again note that the range of $(M_n)_{n \geq 0}$ is exactly $(x_b)_{b \in \T_n}$ from Corollary \ref{cor:wcbush}. Since $(\overline{M}_n)_{n \geq 0}$ is adapted to the same finite filtration as $(M_n)_{n \geq 0}$, (2) implies that the range of $\overline{M}_n$ equals $(y_b)_{b \in \T_n}$ for some $y_b \in \cco{A}$ and $\sup_{b \in \T_n}\|y_b - x_b\| < \delta_n$. Then Corollary \ref{cor:wcbush} implies (3).
\end{proof}

\begin{corollary} \label{cor:final}
For any $A \sbs B_X$ nonempty and nondentable, there exists a countable set $F \sbs \cco{A}$ such that
\begin{enumerate}
\item $\lim_{F \ni f \to \infty} d(f,A) = 0$
\item $F$ is weakly closed and discrete and $\Ext{F} = \emptyset$.
\item $\cco{F}$ has no weak denting point.
\item $\cco{F} \cap \Ext{\overline{co}^{w*}(F)} = \emptyset$.
\end{enumerate}
\end{corollary}

\begin{proof}
Let $A$ be as above. Let $\eps >0$ such that $A$ is non-$\eps$-dentable and let $\delta < \frac{\eps}{2}$. Let $\delta_n$ be any positive, summable sequence, and let $(\A_n)_{n \geq 0}$, $(M_n)_{n \geq 0}$, and $(\overline{M}_n)_{n \geq 0}$ be the filtration, $(\delta_n)_{n \geq 0}$-quasimartingale, and martingale afforded to us by Corollary \ref{cor:wcmartingale}. Let $F \sbs \cco{A}$ be the range of the martingale. Since $(M_n)_{n \geq 0}$ is $A$-valued and $\|M_n-\overline{M}_n\|_\infty < \delta_n$, $\lim_{F \ni f \to \infty} d(f,A) = 0$, showing (1).

By Corollary \ref{cor:wcmartingale}, $F$ is weakly closed and discrete and clearly has no extreme point since it is a $\delta$-separated bush, showing (2).

Since weak denting points of $\cco{F}$ are extreme points, and since $F$ has no extreme points, the set of weak denting points of $\cco{F}$ is contained in $\cco{F} \setminus F$. But since $\cco{F} \setminus F$ is weakly open in $\cco{F}$, it follows that $\cco{F} \setminus F$ contains no weak denting point. This shows (3).

For (4), we first observe that the converse of the Krein-Milman theorem (\cite[Lemma 8.5]{DS}) implies that every extreme point of $\overline{co}^{w*}(F)$ is a weak* denting point of $\overline{co}^{w*}(F)$. To see this, let $x$ be an extreme point of $\overline{co}^{w*}(F)$ and assume $x$ is not a weak* denting point. Then there is an open neighborhood $U \sbs X^{**}$ of $x$ such that $x \in \overline{co}^{w*}\left(\overline{co}^{w*}(F) \setminus U\right)$. Then since $\overline{co}^{w*}(F) \setminus U$ is weak* compact, the converse to Krein-Milman implies every extreme point of $\overline{co}^{w*}\left(\overline{co}^{w*}(F) \setminus U\right)$, in particular $x$, is contained in $\overline{co}^{w*}(F) \setminus U$, a contradiction. Then (4) follows from (3) since weak* denting points of $\cco{F} \cap \overline{co}^{w*}(F) \sbs X^{**}$ are the same as weak denting points of $ \cco{F} \cap \overline{co}^{w*}(F) \sbs X$. 

\end{proof}

\bibliographystyle{amsalpha}
\bibliography{nondentablesetsbib}
\end{document}